\documentclass[12pt]{article}
\usepackage{amsmath,amsfonts,amssymb,amsthm}
\usepackage{enumerate}
\usepackage{caption}
\usepackage{pbox}

\numberwithin{equation}{section}
\theoremstyle{plain}
\newtheorem{theorem}[equation]{Theorem}

\newtheorem{lemma}[equation]{Lemma}
\newtheorem{proposition}[equation]{Proposition}

\theoremstyle{definition}
\newtheorem{definition}[equation]{Definition}

\newtheorem{example}[equation]{Example}

\newtheorem*{openproblem}{Open Problem}

\numberwithin{equation}{section}

\newcommand{\R}{{\mathbb R}}
\newcommand{\N}{{\mathbb N}}

\newcommand{\Om}{\Omega}

\providecommand{\vint}[1]{\mathchoice
          {\mathop{\vrule width 5pt height 3 pt depth -2.5pt
                  \kern -9pt \kern 1pt\intop}\nolimits_{\kern -5pt{#1}}}
          {\mathop{\vrule width 5pt height 3 pt depth -2.6pt
                  \kern -6pt \intop}\nolimits_{\kern -3pt{#1}}}
          {\mathop{\vrule width 5pt height 3 pt depth -2.6pt
                  \kern -6pt \intop}\nolimits_{\kern -3pt{#1}}}
          {\mathop{\vrule width 5pt height 3 pt depth -2.6pt
                  \kern -6pt \intop}\nolimits_{\kern -3pt{#1}}}}

\newcommand{\eps}{\varepsilon}
\newcommand{\loc}{\mathrm{loc}}

\newcommand{\BV}{\mathrm{BV}}
\newcommand{\liploc}{\mathrm{Lip}_{\mathrm{loc}}}

\newcommand{\ch}{\text{\raise 1.3pt \hbox{$\chi$}\kern-0.2pt}}

\DeclareMathOperator{\capa}{Cap}
\DeclareMathOperator{\rcapa}{cap}

\DeclareMathOperator{\dist}{dist}
\DeclareMathOperator{\diam}{diam}

\DeclareMathOperator{\Lip}{Lip}

\DeclareMathOperator{\supp}{spt}

\DeclareMathOperator{\fint}{fine-int}

\begin{document}
\title{A new Cartan-type property and strict quasicoverings when $p=1$ in metric spaces
\footnote{
\hfill \break {\it Keywords\,}: metric measure space, function of bounded variation, fine topology, Cartan property, strict quasicovering, fine Newton-Sobolev space
}}
\author{Panu Lahti}
\maketitle

\begin{abstract}
In a complete metric space that is equipped with a doubling measure
and supports a Poincar\'e inequality, we prove a new Cartan-type property
for the fine topology in the case
$p=1$. Then we use this property to prove the existence of $1$-finely open \emph{strict subsets} and \emph{strict quasicoverings} of $1$-finely open sets. As an application, we study fine
Newton-Sobolev spaces in the case $p=1$, that is,
Newton-Sobolev spaces defined on $1$-finely open sets.
\end{abstract}

\section{Introduction}

Nonlinear fine potential theory in metric spaces has been
studied in several papers in recent years, see \cite{BBL-SS,BBL-CCK,BBL-WC}.
Much of nonlinear potential theory, for $1<p<\infty$,
deals with $p$-harmonic functions, which are
local minimizers of the $L^p$-norm of $|\nabla u|$.
Such minimizers can be defined also in metric measure spaces by using \emph{upper gradients},
and the notion can be extended to the case $p=1$ by considering
functions of least gradient, which are $\BV$ functions that minimize the total variation locally; see Section \ref{preliminaries} for definitions.

Nonlinear fine potential theory is concerned with studying $p$-harmonic functions
and related superminimizers by means of the \emph{p-fine topology}.
For nonlinear fine potential theory and its history in the Euclidean setting, for $1<p<\infty$, see especially the monographs \cite{AH,HKM,MZ}, as well as
the monograph \cite{BB} in the metric setting.
The typical assumptions of a metric space, which we make also in this paper, 
are that the space is complete, equipped with a doubling measure, and supports
a Poincar\'e inequality.

A central result in fine potential theory is the \emph{(weak) Cartan property}
for superminimizer functions.
In \cite{L-WC} we proved the following formulation of this property in the
case $p=1$.

\begin{theorem}[{\cite[Theorem 1.1]{L-WC}}]\label{thm:weak Cartan property}
	Let $A\subset X$ and let $x\in X\setminus A$ such that $A$
	is $1$-thin at $x$.
	Then there exist $R>0$ and $u_1,u_2\in\BV(X)$ that are $1$-superminimizers in $B(x,R)$
	such that $\max\{u_1^{\wedge},u_2^{\wedge}\}=1$ in $A\cap B(x,R)$ and
	$u_1^{\vee}(x)=0=u_2^{\vee}(x)$.
\end{theorem}

In \cite{L-CK} we used this property to prove the so-called Choquet property
concerning finely open and \emph{quasiopen} sets in the case $p=1$, similarly as can be done
when $1<p<\infty$ (see \cite{BBL-CCK}).
On the other hand, it is natural to consider an alternative version of the weak Cartan property.
In the case $p>1$, superminimizers are Newton-Sobolev functions, but in the 
case $p=1$ they are only $\BV$ functions and so
the question arises whether the functions $u_1,u_2$ above can be
replaced by a Newton-Sobolev function (even though it would no longer be
a superminimizer).
In Theorem \ref{thm:Cartan Sobolev version} we show that such a new Cartan-type
property indeed holds.

It is said that a set $A$ is a $p$-\emph{strict} subset of a set $D$ if there exists
a Newton-Sobolev function $u\in N^{1,p}(X)$ such that $u=1$ on $A$ and
$u=0$ on $X\setminus D$.
In \cite{BBL-SS} it was shown that if $U$ is a $p$-finely open set
($1<p<\infty$) and $x\in U$, then
there exists a $p$-finely open strict subset $V\Subset U$ such that $x\in V$.
The proof was based on the weak Cartan property.
In Theorem \ref{thm:strict subsets with Newtonian functions}
we show that the analogous result is true in the case $p=1$. Here we
need the Cartan-type property involving a Newton-Sobolev function
(instead of the $\BV$ superminimizer functions).

This result on the existence of $1$-strict subsets can be combined with the \emph{quasi-Lindel\"of
principle} to prove the existence of \emph{strict quasicoverings} of
$1$-finely open
sets, that is, countable coverings by $1$-finely open strict subsets.
We do this in Proposition \ref{prop:existence of strict quasicovering},
and it is again analogous to the case $1<p<\infty$, see \cite{BBL-SS}.
Such coverings will be useful in future research when considering partition of unity arguments in finely
open sets. In this paper, we apply strict quasicoverings in defining and studying
\emph{fine Newton-Sobolev spaces}, that is, Newton-Sobolev spaces defined on finely open
or quasiopen sets. In the case $1<p<\infty$, these were studied in
\cite{BBL-SS}.
In Section \ref{sec:fine sobolev} we show that the theory we have
developed allows us to prove directly analogous results in the case $p=1$.

\section{Preliminaries}\label{preliminaries}

In this section we introduce the notation, definitions,
and assumptions used in the paper.

Throughout this paper, $(X,d,\mu)$ is a complete metric space that is equip\-ped
with a metric $d$ and a Borel regular outer measure $\mu$ that satisfies
a doubling property, meaning that
there exists a constant $C_d\ge 1$ such that
\[
0<\mu(B(x,2r))\le C_d\mu(B(x,r))<\infty
\]
for every ball $B(x,r):=\{y\in X:\,d(y,x)<r\}$.
We also assume that $X$ supports a $(1,1)$-Poincar\'e inequality defined below,
and that $X$ contains at least 2 points.
For a ball $B=B(x,r)$ and $a>0$, we sometimes abbreviate $aB:=B(x,ar)$;
note that in metric spaces, a ball (as a set) does not necessarily have a unique center and radius,
but we will always understand these to be predetermined for the balls that we consider.
By iterating the doubling condition, we obtain
for any $x\in X$ and any $y\in B(x,R)$ with $0<r\le R<\infty$ that
\begin{equation}\label{eq:homogeneous dimension}
\frac{\mu(B(y,r))}{\mu(B(x,R))}\ge \frac{1}{C_d^2}\left(\frac{r}{R}\right)^{Q},
\end{equation}
where $Q>1$ only depends on the doubling constant $C_d$.
When we want to state that a constant $C$
depends on the parameters $a,b, \ldots$, we write $C=C(a,b,\ldots)$.
When a property holds outside a set of $\mu$-measure zero, we say that it holds
almost everywhere, abbreviated a.e.

As a complete metric space equipped with a doubling measure,
$X$ is proper, that is, closed and bounded sets are compact.
For any $\mu$-measurable set $D\subset X$, we define $\liploc(D)$ to be the
space of functions $u$ on $D$ such that for every $x\in D$ there exists $r>0$
such that $u\in \Lip(D\cap B(x,r))$.
For an open set $\Omega\subset X$, a function $u\in \liploc(\Om)$ is then
in $\Lip(\Om')$ for every open $\Omega'\Subset\Omega$; this notation
means that $\overline{\Omega'}$ is a
compact subset of $\Omega$. Other local spaces of functions are defined analogously.

For any $A\subset X$ and $0<R<\infty$, the restricted Hausdorff content
of codimension one is defined to be
\[
\mathcal{H}_{R}(A):=\inf\left\{ \sum_{i=1}^{\infty}
\frac{\mu(B(x_{i},r_{i}))}{r_{i}}:\,A\subset\bigcup_{i=1}^{\infty}B(x_{i},r_{i}),\,r_{i}\le R\right\}.
\]
The codimension one Hausdorff measure of $A\subset X$ is then defined to be
\[
\mathcal{H}(A):=\lim_{R\rightarrow 0}\mathcal{H}_{R}(A).
\]

All functions defined on $X$ or its subsets will take values in $[-\infty,\infty]$.
By a curve we mean a nonconstant rectifiable continuous mapping from a compact interval of the real line
into $X$.
A nonnegative Borel function $g$ on $X$ is an upper gradient 
of a function $u$
on $X$ if for all curves $\gamma$, we have
\begin{equation}\label{eq:definition of upper gradient}
|u(x)-u(y)|\le \int_\gamma g\,ds,
\end{equation}
where $x$ and $y$ are the end points of $\gamma$
and the curve integral is defined by using an arc-length parametrization,
see \cite[Section 2]{HK} where upper gradients were originally introduced.
We interpret $|u(x)-u(y)|=\infty$ whenever  
at least one of $|u(x)|$, $|u(y)|$ is infinite.

Let $1\le p<\infty$;
we are going to work solely with $p=1$, but we give
definitions that cover all values of $p$ where it takes no extra work.
We say that a family of curves $\Gamma$ is of zero $p$-modulus if there is a 
nonnegative Borel function $\rho\in L^p(X)$ such that 
for all curves $\gamma\in\Gamma$, the curve integral $\int_\gamma \rho\,ds$ is infinite.
A property is said to hold for $p$-almost every curve
if it fails only for a curve family with zero $p$-modulus. 
If $g$ is a nonnegative $\mu$-measurable function on $X$
and (\ref{eq:definition of upper gradient}) holds for $p$-almost every curve,
we say that $g$ is a $p$-weak upper gradient of $u$. 
By only considering curves $\gamma$ in a set $D\subset X$,
we can talk about a function $g$ being a ($p$-weak) upper gradient of $u$ in $D$.

Let $D\subset X$ be a $\mu$-measurable set.
We define the norm
\[
\Vert u\Vert_{N^{1,p}(D)}:=\Vert u\Vert_{L^p(D)}+\inf \Vert g\Vert_{L^p(D)},
\]
where the infimum is taken over all $p$-weak upper gradients $g$ of $u$ in $D$.
The usual Sobolev space $W^{1,p}$ is replaced in the metric setting by the Newton-Sobolev space
\[
N^{1,p}(D):=\{u:\|u\|_{N^{1,p}(D)}<\infty\},
\]
which was first introduced in \cite{S}.
We understand every Newton-Sobolev function to be defined at every $x\in D$
(even though $\Vert \cdot\Vert_{N^{1,p}(D)}$ is then only a seminorm).
It is known that for any $u\in N_{\loc}^{1,p}(D)$, there exists a minimal $p$-weak
upper gradient of $u$ in $D$, always denoted by $g_{u}$, satisfying $g_{u}\le g$ 
a.e. on $D$ for any $p$-weak upper gradient $g\in L_{\loc}^{p}(D)$
of $u$ in $D$, see \cite[Theorem 2.25]{BB}.

For any $D\subset X$, the space of Newton-Sobolev functions with zero boundary values is defined to be
\[
N_0^{1,p}(D):=\{u|_{D}:\,u\in N^{1,p}(X)\textrm{ and }u=0\textrm { on }X\setminus D\}.
\]
This is a subspace of $N^{1,p}(D)$ when $D$ is $\mu$-measurable, and it can always
be understood to be a subspace of $N^{1,p}(X)$.

The $p$-capacity of a set $A\subset X$ is
\[
\capa_p(A):=\inf \Vert u\Vert_{N^{1,p}(X)},
\]
where the infimum is taken over all functions $u\in N^{1,p}(X)$ such that $u\ge 1$ on $A$.
If a property holds outside a set
$A\subset X$ with $\capa_p(A)=0$, we say that it holds $p$-quasieverywhere, or $p$-q.e.
If $D\subset X$ is $\mu$-measurable, then
\begin{equation}\label{eq:quasieverywhere equivalence classes}
\Vert u\Vert_{N^{1,p}(D)}=0\quad \textrm{if}\quad u=0\  p\textrm{-q.e. on }D,
\end{equation}
see \cite[Proposition 1.61]{BB}.

We know that $\capa_p$ is an outer capacity, meaning that
\[
\capa_p(A)=\inf_{\substack{W\textrm{ open} \\A\subset W}}\capa_p(W)
\]
for any $A\subset X$, see e.g. \cite[Theorem 5.31]{BB}.
By \cite[Theorem 4.3, Theorem 5.1]{HaKi}, for any $A\subset X$ it holds that
\begin{equation}\label{eq:null sets of Hausdorff measure and capacity}
\capa_{1}(A)=0\quad\textrm{if and only if}\quad\mathcal H(A)=0.
\end{equation}

We say that a set $U\subset X$ is $p$-quasiopen if for every $\eps>0$
there is an open set $G\subset X$ such that $\capa_p(G)<\eps$ and
$U\cup G$ is open. We say that a function $u$ defined on a
set  $D\subset X$ is $p$-quasicontinuous on $D$ if for every $\eps>0$ there
is an open set $G\subset X$ such that $\capa_p(G)<\eps$ and $u|_{D\setminus G}$
is continuous (as a real-valued function).
It is a well-known fact that Newton-Sobolev functions are quasicontinuous;
for a proof of the following theorem, see \cite[Theorem 1.1]{BBS} or \cite[Theorem 5.29]{BB}.

\begin{theorem}\label{thm:quasicontinuity}
	Let $\Om\subset X$ be open and let $u\in N^{1,p}(\Om)$ (with $1\le p<\infty$).
	Then $u$ is $p$-quasicontinuous on $\Om$.
\end{theorem}

The variational $p$-capacity of a set $A\subset D$
with respect to $D\subset X$ is given by
\[
\rcapa_p(A,D):=\inf \int_X g_u^p \,d\mu,
\]
where the infimum is taken over functions $u\in N_0^{1,p}(D)$ such that
$u\ge 1$ on $A$, and $g_u$ is the minimal $p$-weak upper gradient of $u$ (in $X$).
By truncation, we see that we can also assume that $0\le u\le 1$ on $X$
(and the same applies to the $p$-capacity).
For basic properties satisfied by capacities, such as monotonicity and countable subadditivity, see \cite{BB,BB-VC}.

Next we recall the definition and basic properties of functions
of bounded variation on metric spaces, following \cite{M}. See also
the monographs \cite{AFP, EvaG92, Fed, Giu84, Zie89} for the classical 
theory in the Euclidean setting.
Let $\Om\subset X$ be an open set.
Given $u\in L^1_{\loc}(\Om)$, the total variation of $u$ in $\Om$ is defined
to be
\[
\Vert Du\Vert(\Om):=\inf\left\{\liminf_{i\to\infty}\int_\Om g_{u_i}\,d\mu:\, u_i\in \Lip_{\loc}(\Om),\, u_i\to u\textrm{ in } L^1_{\loc}(\Om)\right\},
\]
where each $g_{u_i}$ is the minimal $1$-weak upper gradient of $u_i$ in $\Om$.
(In \cite{M}, local Lipschitz constants were used instead of upper gradients, but
the properties of the total variation can be proved similarly with either definition.)
We say that a function $u\in L^1(\Om)$ is of bounded variation, 
and denote $u\in\BV(\Om)$, if $\|Du\|(\Om)<\infty$.
For an arbitrary set $A\subset X$, we define
\[
\|Du\|(A):=\inf_{\substack{W\textrm{ open} \\A\subset W}}\|Du\|(W).
\]
If $u\in L^1_{\loc}(\Om)$ and $\Vert Du\Vert(\Omega)<\infty$, then
$\|Du\|(\cdot)$ is a Radon measure on $\Omega$ by \cite[Theorem 3.4]{M}.
A $\mu$-measurable set $E\subset X$ is said to be of finite perimeter if $\|D\ch_E\|(X)<\infty$, where $\ch_E$ is the characteristic function of $E$.
The perimeter of $E$ in $\Omega$ is also denoted by
$P(E,\Omega):=\|D\ch_E\|(\Omega)$.

The lower and upper approximate limits of a function $u$ on $X$
are defined respectively by
\[
u^{\wedge}(x):
=\sup\left\{t\in\R:\,\lim_{r\to 0}\frac{\mu(B(x,r)\cap\{u<t\})}{\mu(B(x,r))}=0\right\}
\]
and
\[
u^{\vee}(x):
=\inf\left\{t\in\R:\,\lim_{r\to 0}\frac{\mu(B(x,r)\cap\{u>t\})}{\mu(B(x,r))}=0\right\}.
\]
Unlike Newton-Sobolev functions, we understand $\BV$ functions to be
$\mu$-equivalence classes. To consider fine properties, we need to
consider the pointwise representatives $u^{\wedge}$ and $u^{\vee}$.

We will assume throughout the paper that $X$ supports a $(1,1)$-Poincar\'e inequality,
meaning that there exist constants $C_P>0$ and $\lambda \ge 1$ such that for every
ball $B(x,r)$, every $u\in L^1_{\loc}(X)$,
and every upper gradient $g$ of $u$,
we have
\begin{equation}\label{eq:poincare inequality}
\vint{B(x,r)}|u-u_{B(x,r)}|\, d\mu 
\le C_P r\vint{B(x,\lambda r)}g\,d\mu,
\end{equation}
where 
\[
u_{B(x,r)}:=\vint{B(x,r)}u\,d\mu :=\frac 1{\mu(B(x,r))}\int_{B(x,r)}u\,d\mu.
\]
The $(1,1)$-Poincar\'e inequality implies the following Sobolev inequality:
if $x\in X$, $0<r<\frac{1}{4}\diam X$, and $u\in N_0^{1,1}(B(x,r))$, then
\begin{equation}\label{eq:sobolev inequality}
\int_{B(x,r)} |u|\,d\mu\le C_S r \int_{B(x,r)}  g_u\,d\mu
\end{equation}
for some constant $C_S=C_S(C_d,C_P)\ge 1$, see \cite[Theorem 5.51]{BB}.
By applying this to approximating functions in the definition of the total
variation, we obtain for any $x\in X$, $0<r<\frac{1}{4}\diam X$,
and any $\mu$-measurable set $E\subset B(x,r)$
\begin{equation}\label{eq:isop inequality with zero boundary values}
\mu(E)\le C_S r P(E,X).
\end{equation}

Next we define the fine topology in the case $p=1$.
\begin{definition}\label{def:1 fine topology}
We say that $A\subset X$ is $1$-thin at the point $x\in X$ if
\[
\lim_{r\to 0}r\frac{\rcapa_1(A\cap B(x,r),B(x,2r))}{\mu(B(x,r))}=0.
\]
We say that a set $U\subset X$ is $1$-finely open if $X\setminus U$ is $1$-thin at every $x\in U$. Then we define the $1$-fine topology as the collection of $1$-finely open sets on $X$ (see \cite[Lemma 4.2]{L-FC} for a proof of the fact that this is indeed a topology).

We denote the $1$-fine interior of a set $H\subset X$, i.e. the largest $1$-finely open set contained in $H$, by $\fint H$. We denote the $1$-fine closure of $H\subset X$, i.e. the smallest $1$-finely closed set containing $H$, by $\overline{H}^1$. We define the \emph{1-base} $b_1 H$ of $H\subset X$
to be the set of points in $X$ where $H$ is \emph{not} $1$-thin.

We say that a function $u$ defined on a set $U\subset X$ is $1$-finely continuous at $x\in U$ if it is continuous at $x$ when $U$ is equipped with the induced $1$-fine topology on $U$ and $[-\infty,\infty]$ is equipped
with the usual topology.
\end{definition}

By \cite[Proposition 6.16]{BB}, for all $x\in X$ and $0<r<\tfrac 18 \diam X$
(in fact, the second inequality holds for all $r>0$)
\begin{equation}\label{eq:variational capacity basic estimate}
\frac{\mu(B(x,r))}{2C_S r}\le \rcapa_1(B(x,r),B(x,2r))\le \frac{C_d \mu(B(x,r))}{r},
\end{equation}
and so obviously $W\subset b_1 W$ for any open set $W\subset X$.

The following fact is given in \cite[Proposition 3.3]{L-Fed}:

\begin{equation}\label{eq:capacity of fine closure}
\capa_1(\overline{A}^1)=\capa_1(A)\quad\textrm{for any }A\subset X.
\end{equation}

The following result describes the close relationship between finely open and
quasiopen sets.

\begin{theorem}[{\cite[Corollary 6.12]{L-CK}}]\label{thm:quasiopen and finely open}
A set $U\subset X$ is $1$-quasiopen if and only if it is the union of a
$1$-finely open set and a $\mathcal H$-negligible set.
\end{theorem}

For an open set $\Om\subset X$,
we denote by $\BV_c(\Om)$ the class of functions $\varphi\in\BV(\Om)$ with compact
support in $\Om$, that is, $\supp \varphi\Subset \Om$.

\begin{definition}
	We say that $u\in\BV_{\loc}(\Om)$ is a 
	$1$-minimizer  in $\Om$ if
	for all $\varphi\in \BV_c(\Om)$,
	\begin{equation}\label{eq:definition of 1minimizer}
	\Vert Du\Vert(\supp\varphi)\le \Vert D(u+\varphi)\Vert(\supp\varphi).
	\end{equation}
	We say that $u\in\BV_{\loc}(\Om)$ is a $1$-superminimizer in $\Om$
	if \eqref{eq:definition of 1minimizer} holds for all nonnegative $\varphi\in \BV_c(\Om)$.
\end{definition}

More precisely, we should talk about $\supp |\varphi|^{\vee}$, since
$\varphi$ is only a.e. defined.
In the literature, $1$-minimizers are usually called functions of least gradient.

\section{A new Cartan-type property}

In this section we prove the new Cartan-type property, given in 
Theorem \ref{thm:Cartan Sobolev version}.
First we take note of a few results that we will need in the proofs;
the following is given in \cite[Lemma 11.22]{BB}.

\begin{lemma}\label{lem:capacity wrt different balls}
	Let $x\in X$, $r>0$, and $A\subset B(x,r)$. Then for every $1<s<t$
	with $tr<\frac 14 \diam X$, we have
	\[
	\rcapa_1(A,B(x,tr))\le \rcapa_1(A,B(x,sr))\le C_S\left(1+\frac{t}{s-1}\right)\rcapa_1(A,B(x,tr)),
	\]
	where $C_S$ is the constant from
	the Sobolev inequality \eqref{eq:sobolev inequality}.
\end{lemma}

\begin{theorem}[{\cite[Theorem 3.16]{L-WC}}]\label{thm:superminimizers are lsc}
	Let $u$ be a $1$-superminimizer in an open set $\Om\subset X$. Then $u^{\wedge}\colon\Om\to (-\infty,\infty]$
	is lower semicontinuous.
\end{theorem}

As mentioned in the introduction, in \cite{L-WC} we proved
a weak Cartan property for $p=1$, more precisely in the following form.

\begin{theorem}[{\cite[Theorem 5.3]{L-WC}}]\label{thm:weak Cartan property in text}
	Let $A\subset X$ and let $x\in X\setminus A$ be such that $A$
	is $1$-thin at $x$.
	Then there exist $R>0$ and $E_0,E_1\subset X$ such that $\ch_{E_0},\ch_{E_1}\in\BV(X)$,
	$\ch_{E_0}$ and $\ch_{E_1}$ are $1$-superminimizers in $B(x,R)$,
	$\max\{\ch_{E_0}^{\wedge},\ch_{E_1}^{\wedge}\}=1$ in $A\cap B(x,R)$,
	$\ch_{E_0}^{\vee}(x)=0=\ch_{E_1}^{\vee}(x)$,
	$\{\max\{\ch_{E_0}^{\vee},\ch_{E_1}^{\vee}\}>0\}$ is $1$-thin at $x$,
	and
	\[
	\lim_{r\to 0}r\frac{P(E_0,B(x,r))}{\mu(B(x,r))}=0,\qquad
	\lim_{r\to 0}r\frac{P(E_1,B(x,r))}{\mu(B(x,r))}=0.
	\]
\end{theorem}

Now we collect a few facts that are not included in the above statement,
but follow from the proof given in \cite{L-WC}.
Defining $B_j:=B(x,2^{-j}R)$
and $H_j:=B_j\setminus \tfrac{9}{10}\overline{B_{j+1}}$ for $j=0,1,\ldots$,
there exists an open set $W\supset A$ that is $1$-thin at $x$,
\begin{equation}\label{eq:E0 and E1 contain...}
W\cap \bigcup_{j=0,2,\ldots}H_j\subset E_0\quad\textrm{and}\quad
W\cap \bigcup_{j=1,3,\ldots}H_j\subset E_1,
\end{equation}
and
\begin{equation}\label{eq:E0 and E1 included in annuli}
\begin{split}
E_0\subset \left(\tfrac 32 B_{0}\setminus 
\tfrac 45 B_{1}\right)\cup \bigcup_{j=2,4,\ldots}^{\infty}\tfrac 54 B_{j}\setminus 
\tfrac 45 B_{j+1}\quad\textrm{and}\\
E_1\subset \left(\tfrac 32 B_{1}\setminus 
\tfrac 45 B_{2}\right)\cup 
\bigcup_{j=3,5,\ldots}^{\infty}\tfrac 54 B_{j}\setminus 
\tfrac 45 B_{j+1}.
\end{split}
\end{equation}
Moreover, by \cite[Eq (5.6)]{L-WC}, for all $i=2,4,6,\ldots$ we have
\begin{equation}\label{eq:perimeter estimate for E0}
P(E_0\cap \tfrac 54 B_i,X) \le 5C_S \rcapa_1(W\cap B_i,2B_i),
\end{equation}
and similarly for all $i=3,5,7,\ldots$,
\begin{equation}\label{eq:perimeter estimate for E1}
P(E_1\cap \tfrac 54 B_i,X) \le 5C_S \rcapa_1(W\cap B_i,2B_i).
\end{equation}
From the proof it can also be seen that if $R>0$ is chosen to be smaller, all of the above results still hold. The same will then apply to the conclusion of
the next lemma.
Let $B_j$ and $H_j$ be defined as above.

\begin{lemma}\label{lem:sets Fj in Cartan property}
	Let $A\subset X$ and let $x\in X\setminus A$ be such that $A$
	is $1$-thin at $x$. Then there exists a number $R>0$,
	an open set
	$W\supset A$ that is $1$-thin at $x$, and
	open sets $F_j\supset W\cap H_j$ such that 
	$F_j\subset \tfrac 54 B_j\setminus \tfrac 34 B_{j+1}$
	for all $j=0,1,\ldots$, and
	\begin{equation}\label{eq:perimeter sum of Fj}
	\sum_{j=i}^{\infty}P(F_j,X)\le 50C_S^2 \rcapa_1(W\cap B_i,2B_i)
	\end{equation}
	for all $i=0,1,\ldots$.
\end{lemma}

\begin{proof}
	By using the weak Cartan property (Theorem \ref{thm:weak Cartan property in text}),
	choose $R>0$ and $E_0,E_1\subset X$ such that $\ch_{E_0},\ch_{E_1}\in\BV(X)$
	and $\ch_{E_0}$ and $\ch_{E_1}$ are $1$-super\-minimizers in $B(x,R)$.
	We can assume that $R<\tfrac 12\diam X$.
	Also let $W\supset A$ be an open set that is $1$-thin at $x$, as described above. Define
	\[
	F_j:=\{\ch_{E_0}^{\wedge}>0\}\cap \tfrac 54 B_j\setminus \tfrac 34 \overline{B_{j+1}}
	\quad\textrm{for } j=2,4,\ldots,
	\]
	and
	\[
	F_j:=\{\ch_{E_1}^{\wedge}>0\}\cap \tfrac 54 B_j\setminus \tfrac 34 \overline{B_{j+1}}
	\quad\textrm{for } j=3,5,\ldots.
	\]
	By \eqref{eq:E0 and E1 contain...}, we have $F_j\supset W\cap H_j$
	for all $j=2,3,\ldots$ as desired.
	The sets $F_j$ are open by Theorem \ref{thm:superminimizers are lsc}.
	By Lebesgue's differentiation theorem, the sets $\{\ch_{E_0}^{\wedge}>0\}$
	and $\{\ch_{E_1}^{\wedge}>0\}$ differ from $E_0$ and $E_1$, respectively, only by a
	set of $\mu$-measure zero.
	Thus by \eqref{eq:E0 and E1 included in annuli} and the fact that the sets $F_j$ are
	at a positive distance from each other, we find that for all $i=2,4,\ldots$,
	\[
	P(E_0\cap \tfrac 54 B_i,X)=P\left(\bigcup_{j=i,i+2,\ldots}F_j,X\right)
	=\sum_{j=i,i+2,\ldots}P(F_j,X),
	\]
	and similarly for all $i=3,5,\ldots$,
	\[
	P(E_1\cap \tfrac 54 B_i,X)=\sum_{j=i,i+2,\ldots}P(F_j,X).
	\]
	Combining these with \eqref{eq:perimeter estimate for E0} and 
	\eqref{eq:perimeter estimate for E1}, and using
	Lemma \ref{lem:capacity wrt different balls}, we have for all
	$i=2,3,\ldots$
	\begin{align*}
	\sum_{j=i}^{\infty}P(F_j,X)
	&\le 5C_S (\rcapa_1(W\cap B_i,2B_i)+\rcapa_1(W\cap B_{i+1},2B_{i+1}))\\
	&\le 5C_S (\rcapa_1(W\cap B_i,2B_i)+5C_S\rcapa_1(W\cap B_{i+1},4B_{i+1}))\\
	&\le 25C_S^2 (\rcapa_1(W\cap B_i,2B_i)+\rcapa_1(W\cap B_{i},4B_{i+1}))\\
	&= 50C_S^2 \rcapa_1(W\cap B_i,2B_i).
	\end{align*}
	Then by replacing $R$ with $R/4$, we have the result.
\end{proof}

Recall the constant $\lambda \ge 1$ from the Poincar\'e inequality
\eqref{eq:poincare inequality}.
We have the following boxing inequality from \cite[Theorem 3.1]{KKST2}.
Note that in \cite{KKST2} it is assumed that $\mu(X)=\infty$, but the proof reveals that
we can alternatively assume $\mu(F)<\mu(X)/2$.
\begin{theorem}\label{thm:boxing inequality}
Let $F\subset X$ be an open set of finite perimeter with $\mu(F)<\mu(X)/2$ (in particular,
$\mu(F)$ is finite). Then there exists a collection of balls
	$\{B_k=B(x_k,r_k)\}_{k\in\N}$
	such that the balls $\lambda B_k$ are disjoint,
	$F\subset \bigcup_{k=1}^{\infty}5\lambda B_k$,
	\[
	\frac{1}{2C_d}\le \frac{\mu(B_k\cap F)}{\mu(B_k)}
	\le \frac{1}{2}
	\]
	for all $k\in\N$, and
	\[
	\sum_{k=1}^{\infty}\frac{\mu(5\lambda B_k)}{5\lambda r_k}\le C_B P(F,X)
	\]
	for some constant $C_B=C_B(C_d,C_P,\lambda)$.
\end{theorem}

Now we can show the following Cartan-type property.

\begin{theorem}\label{thm:Cartan Sobolev version}
	Let $A\subset X$ and let $x\in X\setminus A$ be such that $A$
	is $1$-thin at $x$. Then there exists a number
	$R>0$, open sets $G\subset V\subset X$,
	and a function $\eta\in N_0^{1,1}(V)$
	such that $A\cap B(x,R)\subset G$, $V$ is $1$-thin at $x$,
	$0\le\eta\le 1$ on $X$, $\eta=1$ on $G$, and
	\begin{equation}\label{eq:estimate for Cartan function}
	\lim_{r\to 0}\frac{r}{\mu(B(x,r))}\Vert \eta\Vert_{N^{1,1}(B(x,r))}=0.
	\end{equation}
\end{theorem}

\begin{proof}
	Take $R>0$, an open set $W\supset A$, and open sets
	$F_j\subset \tfrac 54 B_j\setminus \tfrac 34 B_{j+1}$ as given by
	Lemma \ref{lem:sets Fj in Cartan property}.
	Let
	\[
	\delta:=\frac{1}{2^8 (680 \lambda)^Q C_d^{3} C_S^3},
	\]
	where $Q>1$ is the exponent in \eqref{eq:homogeneous dimension}. 
	We can assume that $R\le \min\left\{1,\tfrac 18\diam X\right\}$.
	Since $\mu(\{x\})=0$ (see \cite[Corollary 3.9]{BB}),
	we can also assume $R$ to be so small that
	$\mu(\tfrac 54 B_0)<\tfrac 12 \mu(X)$, and so also $\mu(F_j)<\tfrac 12 \mu(X)$ for all
	$j=0,1,\ldots$.
	Since $W$ is $1$-thin at $x$, we can further
	assume that $R$ is so small that
	\begin{equation}\label{eq:W is thin}
	2^{-j}R\frac{\rcapa_1(W\cap B_j,2B_j)}{\mu(B_j)}<\delta
	\end{equation}
	for all $j=0,1,\ldots$. Fix $j$.
	By the boxing inequality
	(Theorem \ref{thm:boxing inequality})
	we find a collection of balls $\{B^j_k=B(x_k^j,r_k^j)\}_{k=1}^{\infty}$
	such that the balls $\lambda B^j_k$ are disjoint,
	$F_j\subset \bigcup_{k=1}^{\infty}5\lambda B_k^j$,
	\begin{equation}\label{eq:Bkj and Fj measure}
	\frac{1}{2C_d}\le \frac{\mu(B_k^j\cap F_j)}{\mu(B_k^j)}
	\le \frac{1}{2}
	\end{equation}
	for all $k\in\N$, and
	\begin{equation}\label{eq:boxing inequality estimate}
	\sum_{k=1}^{\infty}\frac{\mu(5\lambda B_k^j)}{5\lambda r_k^j}\le C_B P(F_j,X).
	\end{equation}
	Thus we have
	\begin{align*}
	\mu(B_k^j) \le 2C_d\mu(B_k^j\cap F_j)
	&\le 2C_d\mu(F_j)\\
	&\le 2^{2-j} R C_d C_S P(F_j,X)\quad\textrm{by }\eqref{eq:isop inequality with zero boundary values}\\
	&\le 2^{8-j} R C_d C_S^3 \rcapa_1(W\cap B_j,2B_j)\quad\textrm{by }\eqref{eq:perimeter sum of Fj}.
	\end{align*}
	Thus for all $k\in\N$,
	\begin{equation}\label{eq:measure ratio of Bkj and Bj}
	\frac{\mu(B_k^j)}{\mu(B_j)}\le 2^{8} C_d C_S^3 2^{-j}R
	\frac{\rcapa_1(W\cap B_j,2B_j)}{\mu(B_j)}
	\le 2^{8} C_d C_S^3\delta
	\end{equation}
	by \eqref{eq:W is thin}. By \eqref{eq:Bkj and Fj measure} we necessarily have
	$F_j\cap B_k^j\neq \emptyset$ for all $k\in\N$,
	and so $\tfrac 54 B_j\cap B_k^j\neq \emptyset$.
	Now if $r_k^j\ge 2^{-j}R$ for some $k\in\N$, then $B_j\subset 4 B_k^j$ and so
	\[
	\frac{\mu(B_k^j)}{\mu(B_j)}\ge
	\frac{1}{C_d^{2}},
	\]
	contradicting \eqref{eq:measure ratio of Bkj and Bj} by our choice of $\delta$.
	Thus $r_k^j\le 2^{-j}R$ for all $k\in\N$, so that
	$x_k^j\in 3 B_j$, and thus by \eqref{eq:homogeneous dimension},
	\[
	\left(\frac{r_k^j}{2^{-j+2}R}\right)^{Q}
	\le C_d^2 \frac{\mu(B_k^j)}{\mu(4 B_j)}
	\le C_d^{2} \frac{\mu(B_k^j)}{\mu(B_j)}
	\le 2^{8} C_d^{3} C_S^3 \delta
	\]
	by \eqref{eq:measure ratio of Bkj and Bj}, so that by our choice of $\delta$,
	\begin{equation}\label{eq:estimate for rkj}
	r_k^j\le (2^{8} C_d^3 C_S^3 \delta)^{1/Q} 2^{-j+2}R=\frac {2^{-j}R}{170 \lambda}.
	\end{equation}
	Thus recalling that $F_j\cap B_k^j\neq \emptyset$,
	so that $(\tfrac 54 B_j\setminus \tfrac 34 B_{j+1})\cap B_k^j \neq \emptyset$,
	we conclude that $20\lambda B_k^j\subset B_{j-1}\setminus B_{j+2}$
	(let $B_{-1}:=B(x,2R)$).
	Now, define Lipschitz functions
	\[
	\xi_k^j:=\max\left\{0,1-\frac{\dist(\cdot,10\lambda B_k^j)}{10\lambda r_k^j}\right\},
	\quad k\in\N,
	\]
	so that $\xi_k^j=1$ on $10\lambda B_k^j$ and $\xi_k^j=0$ on
	$X\setminus 20\lambda B_k^j$. Using the basic properties of $1$-weak upper gradients, see \cite[Corollary 2.21]{BB}, we obtain
	\[
	\int_X g_{\xi_k^j}\,d\mu\le \frac{\mu(20\lambda B_k^j)}{10\lambda r_k^j}.
	\]
	Define $V:=\bigcup_{j=0}^{\infty}\bigcup_{k=1}^{\infty}10\lambda B_k^j$.
	Now for every $i=1,2,\ldots$,
	\begin{equation}\label{eq:estimating capacity of V cap Bi}
	\begin{split}
	\rcapa_1(V\cap B_i,4B_i)
	&\le \rcapa_1\left(\bigcup_{j=i-1}^{\infty}\bigcup_{k=1}^{\infty}10\lambda B_k^j,4B_i\right)\\
	&\le \sum_{j=i-1}^{\infty}\sum_{k=1}^{\infty}
	\rcapa_1\left(10\lambda B_k^j,4B_i\right)\\
	&\le \sum_{j=i-1}^{\infty}\sum_{k=1}^{\infty}\int_X g_{\xi_k^j}\,d\mu\\
	&\le \sum_{j=i-1}^{\infty}\sum_{k=1}^{\infty} \frac{\mu(20\lambda B_k^j)}{10\lambda r_k^j}\\
	&\le C_d^2 C_B \sum_{j=i-1}^{\infty} P(F_j,X)\quad\textrm{by }
	\eqref{eq:boxing inequality estimate}\\
	&\le 50 C_d^2 C_B C_S^2 \rcapa_1(W\cap B_{i-1},2B_{i-1})\quad\textrm{by }
	\eqref{eq:perimeter sum of Fj}.
	\end{split}
	\end{equation}
	Thus
	\[
	2^{-i}R\frac{\rcapa_1(V\cap B_i,4B_i)}{\mu(B_i)}\le 
	50 C_d^3 C_B C_S^2 2^{-i+1}R\frac{\rcapa_1(W\cap B_{i-1},2B_{i-1})}{\mu(B_{i-1})}\to 0
	\]
	as $i\to\infty$, since $W$ is $1$-thin at $x$.
	By Lemma \ref{lem:capacity wrt different balls}
	it is then straightforward to show that $V$ is also $1$-thin at $x$.
	Let us also define the Lipschitz functions
	\[
	\eta_k^j:=\max\left\{0,1-\frac{\dist(\cdot,5\lambda B_k^j)}{5\lambda r_k^j}\right\},
	\quad j=0,1,\ldots,\ k=1,2,\ldots,
	\]
	so that $\eta_k^j=1$ on $5\lambda B_k^j$
	and $\eta_k^j=0$ on $X\setminus 10\lambda B_k^j$, and then let
	\[
	\eta:=\sup_{j=0,1,\ldots,\ k=1,2,\ldots}\eta_k^j.
	\]
	Recall from Lemma \ref{lem:sets Fj in Cartan property} that
	$\bigcup_{j=0}^{\infty} F_j\supset W\cap B(x,R)$;
	thus $\eta\ge 1$ on
	\[
	G:=\bigcup_{j=0}^{\infty}\bigcup_{k=1}^{\infty}5\lambda B_k^j 
	\supset \bigcup_{j=0}^{\infty} F_j\supset W\cap B(x,R)
	\supset A\cap B(x,R).
	\]
	By \cite[Lemma 1.52]{BB} we know that
	$g_{\eta}\le \sum_{j=0}^{\infty}\sum_{k=1}^{\infty}g_{\eta_k^j}$.
	Thus for any $i=1,2,\ldots$,
	\begin{align*}
	\int_{B_i} g_{\eta}\,d\mu
	\le \sum_{j=0}^{\infty}\sum_{k=1}^{\infty}\int_{B_i} g_{\eta_k^j}\,d\mu
	&\le \sum_{j=i-1}^{\infty}\sum_{k=1}^{\infty}\int_{X} g_{\eta_k^j}\,d\mu\\
	&\le 50 C_d C_B C_S^2\rcapa_1(W\cap B_{i-1},2B_{i-1}),
	\end{align*}
	where the last inequality follows just as in the last
	four lines of \eqref{eq:estimating capacity of V cap Bi}.
	Since we assumed $R\le 1$ and so $5\lambda r_k^j\le 1$
	by \eqref{eq:estimate for rkj}, we similarly get
	\[
	\Vert \eta\Vert_{L^1(B_i)}\le 
	50 C_d C_B C_S^2\rcapa_1(W\cap B_{i-1},2B_{i-1}).
	\]
	Using the fact that $W$ is $1$-thin at $x$ and the doubling property of $\mu$, we get
	\eqref{eq:estimate for Cartan function}.
	Estimating just as in the last
	four lines of \eqref{eq:estimating capacity of V cap Bi}, now with $i=1$,
	we get
	\[
	\int_{X} g_{\eta}\,d\mu
	\le \sum_{j=0}^{\infty}\sum_{k=1}^{\infty}\int_{X} g_{\eta_k^j}\,d\mu
	\le 50 C_d C_B C_S^2\rcapa_1(W\cap B_{0},2B_{0})<\infty.
	\]
	Thus $\eta\in N^{1,1}(X)$. Clearly $\eta=0$ on $X\setminus V$, and so $\eta\in N_0^{1,1}(V)$.
	
\end{proof}

\section{$1$-strict subsets}

In this section we study $1$-strict subsets which are defined as follows.

\begin{definition}
A set $A\subset D$ is a \emph{1-strict subset} of $D\subset X$
if there is a function $u\in N_0^{1,1}(D)$ such that $u=1$ on $A$.
\end{definition}

Equivalently, $A$ is a $1$-strict subset of $D$ if $\rcapa_1(A,D)<\infty$.
In \cite[Proposition 6.7]{L-CK} we proved the following result by using the
weak Cartan property (Theorem \ref{thm:weak Cartan property in text}).

\begin{proposition}\label{prop:strict subsets}
	Let $U\subset X$ be $1$-finely open and let $x\in U$. Then there exists
	a $1$-finely open set $W$ such that $x\in W\subset U$,
	and a function $w\in\BV(X)$ such that $0\le w\le 1$ on $X$,
	$w^{\wedge}=1$ on $W$, and $\supp w\Subset U$.
\end{proposition}

This kind of formulation
is sufficient for some purposes, but now we are able to improve it by
replacing $w\in\BV(X)$ with $w\in N^{1,1}(X)$.
The following is our main result on the existence of $1$-strict subsets.

\begin{theorem}\label{thm:strict subsets with Newtonian functions}
	Let $U\subset X$ be $1$-finely open and let $x\in U$. Then there exists
	a $1$-finely open set $W$ such that $x\in W\subset U$,
	and a function $w\in N_0^{1,1}(U)$ such that $0\le w\le 1$ on $X$,
	$w=1$ on $W$, and $\supp w\Subset U$.
	
	Moreover, if $\capa_1(\{x\})=0$, then $\Vert w\Vert_{N^{1,1}(X)}$
	can be made arbitrarily
	small.
\end{theorem}

\begin{proof}
Applying Theorem \ref{thm:Cartan Sobolev version} with the choice
$A=X\setminus U$,
we find a number $R>0$, open sets $G\subset V\subset X$,
and a function $\eta\in N_0^{1,1}(V)$
such that $B(x,R)\subset G\cup U$, $V$ is $1$-thin at $x$,
$0\le \eta\le 1$ on $X$, $\eta=1$ on $G$, and
\[
\lim_{r\to 0}\frac{r}{\mu(B(x,r))}\Vert \eta\Vert_{N^{1,1}(B(x,r))}=0.
\]
Choose $0<r\le R$ such that $r\Vert \eta\Vert_{N^{1,1}(B(x,r))}/\mu(B(x,r))\le 1$ and let
\[
\rho:=\max\left\{0,1-\frac{4\dist(\cdot,B(x,r/2))}{r}\right\}\in \Lip(X),
\]
so that $0\le \rho\le 1$ on $X$, $\rho=1$ on $B(x,r/2)$,
and $\supp \rho\Subset B(x,r)$.
Then let $w:=(1-\eta) \rho$.
By the Leibniz rule (see \cite[Theorem 2.15]{BB}), we have $w\in N^{1,1}(X)$ and
\begin{align*}
\int_X g_{w}\,d\mu
=\int_{B(x,r)} g_{w}\,d\mu
&\le \int_{B(x,r)} g_{\eta}\,d\mu+\int_{B(x,r)} g_{\rho}\,d\mu\\
&\le \frac{\mu(B(x,r))}{r}+\frac{4\mu(B(x,r))}{r}.
\end{align*}
Thus $\Vert w\Vert_{N^{1,1}(X)}\le (5/r+1)\mu(B(x,r))$.
If $\capa_1(\{x\})=0$, then also $\mathcal H(\{x\})=0$
by \eqref{eq:null sets of Hausdorff measure and capacity},
and so we can make $\mu(B(x,r))/r$ as small as we like by choosing suitable $r$.
Then we can also make $\Vert w\Vert_{N^{1,1}(X)}$ arbitrarily small.

Regardless of the value of $\capa_1(\{x\})$,
the set $V$ is $1$-thin at $x$, that is, $x\notin b_1 V$.
Since $V$ is open we have $V\subset b_1 V$; recall
\eqref{eq:variational capacity basic estimate} and the comment after it.
We know that $\overline{V}^1=V\cup b_1 V$ by
\cite[Corollary 3.5]{L-Fed}, so in conclusion $x\notin \overline{V}^1$.
Thus
\[
W:=B(x,r/2)\setminus \overline{V}^1\subset \{w=1\}
\]
is a $1$-finely open neighborhood of $x$. Finally, $\supp w$ is compact and
\[
\supp w\subset \supp \rho\setminus G\subset (U\cup G)\setminus G \subset U,
\]
so that $\supp w\Subset U$. Clearly now $w\in N_0^{1,1}(U)$.
\end{proof}

Let us make a few more observations concerning $1$-strict subsets.
In general it is not clear which subsets $A$ of a set $D$ are $1$-strict subsets.
If $A$ is a compact subset of an open set $D$, we obviously have
$\rcapa_{1}(A,D)<\infty$, and the test function can even be chosen to
be Lipschitz. When $A$ is a compact subset of a $1$-\emph{quasiopen}
set $D$, we cannot necessarily choose a Lipschitz test function
but one might nonetheless suspect that $\rcapa_1(A,D)<\infty$.
However, this is not always the case.

\begin{example}\label{ex:infinite capacity of compact set wrt quasiopen set}
	Let $X=\R^2$ (unweighted), denote the origin by $0$, and let
	\[
	A:=\bigcup_{j=1}^{\infty}A_j\cup \{0\}\quad\textrm{with}\quad
	A_j:=\{2^{-j}\}\times [-1/(2j),1/(2j)].
	\]
	Denoting $A_j^{\eps}:=\{x\in X:\,\dist(x,A_j)<\eps\}$, with $\eps>0$, let
	\[
	D:=\bigcup_{j=1}^{\infty}D_j\cup \{0\}\quad\textrm{with}\quad D_j:=A_j^{2^{-3j}}.
	\]
	Since all the sets $D_j$ are disjoint, it is straightforward to check that
	\[
	\rcapa_{1}(A,D)=\sum_{j=1}^{\infty}\rcapa_{1}(A_j,D_j)
	=\sum_{j=1}^{\infty}\frac 1j=\infty.
	\]
	Now $A$ is clearly a compact set, and $D$ is $1$-quasiopen since
	$D\cup B(0,r)$ is an open set for every $r>0$.
\end{example}

One can also make the sets $A,D$ connected by adding the line $(0,1/2]\times\{0\}$
to $A$, and by adding e.g. the sets $(2^{-j-1},2^{-j})\times (-2^{-j-1},2^{-j-1})$
to $D$; then we still have $\rcapa_{1}(A,D)=\infty$ but the calculation is somewhat
more complicated.

The variational $1$-capacity is an outer capacity in the following weak sense.

\begin{proposition}\label{prop:quasiouter capacity}
Let $A\subset D\subset X$. Then
\[
\rcapa_1(A,D)=\inf_{\substack{V\textrm{ $1$-quasiopen} \\A\subset
V\subset D}}\rcapa_{1}(V,D).
\]
\end{proposition}

\begin{proof}
We can assume that $\rcapa_1(A,D)<\infty$. Fix $0<\eps<1$. Take $u\in N_0^{1,1}(D)$
such that $u=1$ on $A$ and $\int_X g_u\,d\mu<\rcapa_1(A,D)+\eps$.
The set $V:=\{u>1-\eps\}$ is $1$-quasiopen by Theorem \ref{thm:quasicontinuity}, and
\[
\rcapa_1(V,D) \le \int_X g_{u/(1-\eps)}\,d\mu
= \frac{\int_X g_u\,d\mu}{1-\eps}
\le \frac{\rcapa_1(A,D)+\eps}{1-\eps}.
\]
Since $0<\eps<1$ was arbitrary, we have the result.
\end{proof}

Even though $1$-quasiopen sets and $1$-finely open sets are very closely
related (recall Theorem \ref{thm:quasiopen and finely open}), it is not
clear whether the following holds.

\begin{openproblem}
If $D\subset X$ and $A\subset \fint D$, do we have
\[
\rcapa_1(A,D)=\inf_{\substack{V\textrm{ $1$-finely open} \\A\subset
		V\subset D}}\rcapa_{1}(V,D)?
\]
\end{openproblem}

Note that according to Theorem \ref{thm:strict subsets with Newtonian functions},
the above property does hold in the very special case when $A$ is a point
with $1$-capacity zero.

Let us say that a set $K\subset X$ is \emph{1-quasiclosed} if $X\setminus K$ is $1$-quasiopen. Now we can show that $1$-strict subsets have the following
continuity.

\begin{proposition}
Let $D\subset X$ and let $K_1\supset K_2\supset \ldots$ be bounded
$1$-quasiclosed subsets
of $D$ such that $\rcapa_1(K_1,D)<\infty$. Then for
$K:=\bigcap_{i=1}^{\infty}K_i$ we have
\[
\rcapa_1(K,D)=\lim_{i\to\infty}\rcapa_1(K_i,D).
\]
\end{proposition}

We will show in Example \ref{ex:necessity of finite capacity of compact set}
below that the assumption $\rcapa_1(K_1,D)<\infty$ is needed.

\begin{proof}
Fix $\eps>0$. By Proposition \ref{prop:quasiouter capacity}
we find a $1$-quasiopen set $V$ such that
$K\subset V\subset D$ and $\rcapa_1(V,D)<\rcapa_1(K,D)+\eps$.
For each $j\in\N$ we find an open set $\widetilde{G}_j\subset X$ such that
$V\cup \widetilde{G}_j$ is open and $\capa_1(\widetilde{G}_j)\to 0$ as $j\to\infty$.
For each $i,j\in\N$, we find an open set $G_{i,j}\subset X$ such that
$K_i\setminus G_{i,j}$ is compact and $\capa_1(G_{i,j})<2^{-i-j}$.
Letting $G_j:=\widetilde{G}_j\cup \bigcup_{i=1}^{\infty}G_{i,j}$ for each $j\in\N$, we have that each $V\cup G_j$ is open, each $K_i\setminus G_j$ is 
compact, and $\capa_1(G_j)\to 0$ as $j\to \infty$. Then for each
$j\in\N$ we find a function $w_j\in N^{1,1}(X)$ such that
$0\le w_j\le 1$ on $X$, $w_j=1$ on $G_j$, and
$\Vert w_j\Vert_{N^{1,1}(X)}\to 0$ as $j\to\infty$.
Passing to a subsequence (not relabeled), we can assume that $w_j\to 0$ a.e.

Since $\rcapa_1(K_1,D)<\infty$, we find
$v\in N_0^{1,1}(D)$ such that $0\le v\le 1$ on $X$ and $v=1$ on $K_1$. For each $j\in\N$, let $\rho_j:=vw_j$. Then
$\Vert \rho_j\Vert_{L^1(X)}\to 0$ as $j\to\infty$, and by the Leibniz rule
(see \cite[Theorem 2.15]{BB}),
\[
\int_X g_{\rho_j}\,d\mu\le \int_X g_{w_j}\,d\mu+\int_X w_j g_{v}\,d\mu
\to 0
\]
as $j\to\infty$; for the second term this follows
from Lebesgue's dominated convergence theorem.
Thus $\rcapa_1(G_j\cap K_1,D)\to 0$. Fix $j\in\N$ such that
$\rcapa_1(G_j\cap K_1,D)<\eps$.
Since every $K_i\setminus G_j$ is compact and $V\cup G_j$ is open,
for some $i\in\N$ we have $K_i\setminus G_j\subset V\cup G_j$.
Thus $K_i\subset V\cup G_j$.
Then
\begin{align*}
\rcapa_1(K_i,D)
&\le \rcapa_1(V\cup (G_j\cap K_1),D)\\
&\le \rcapa_1(V,D)+\rcapa_1(G_j\cap K_1,D)\\
&\le \rcapa_1(K,D)+\eps+\rcapa_1(G_j\cap K_1,D)\\
&\le \rcapa_1(K,D)+2\eps.
\end{align*}
Since $\eps>0$ was arbitrary, the proof is concluded.
\end{proof}

\begin{example}\label{ex:necessity of finite capacity of compact set}
	In the notation of Example
	\ref{ex:infinite capacity of compact set wrt quasiopen set},
	let $K_i:=\bigcup_{j=i}^{\infty}A_j\cup \{0\}$ for each $i\in\N$.
	These are compact
	sets and similarly as in Example
	\ref{ex:infinite capacity of compact set wrt quasiopen set}
	we find that $\rcapa_{1}(K_i,D)=\infty$ for every
	$i\in\N$. However, $\rcapa_1(K,D)=0$ for $K:=\bigcap_{i=1}^{\infty}K_i=\{0\}$,
	by the fact that a point has $1$-capacity zero and by using
	\eqref{eq:quasieverywhere equivalence classes}.
\end{example}

\section{Application to fine Sobolev spaces}\label{sec:fine sobolev}

Bj\"orn--Bj\"orn--Latvala \cite{BBL-SS} have studied different definitions of
Newton-Sobolev spaces on quasiopen sets in metric spaces in the case $1<p<\infty$.
As an application of the theory we have developed, we show that the analogous
results hold for $p=1$.

First we prove the following fact in a very similar way as
it is proved in the case $1<p<\infty$, see \cite[Theorem 1.4(b)]{BBL-CCK}
and \cite[Theorem 4.9(b)]{BBL-WC}. Recall that a function $u$ defined on a set  $U\subset X$ is $1$-quasicontinuous on $U$ if for every $\eps>0$ there is an open set $G\subset X$ such that
$\capa_1(G)<\eps$ and $u|_{U\setminus G}$ is continuous (as a real-valued function).

\begin{theorem}\label{thm:fine continuity and quasicontinuity equivalence}
	A function $u$ on a $1$-quasiopen set $U$ is
	$1$-quasicontinuous on $U$ if and only if it is finite $1$-q.e. and $1$-finely
	continuous $1$-q.e. on $U$.
\end{theorem}

\begin{proof}
	To prove one direction, suppose there is a set $N\subset U$ such that $\capa_1(N)=0$ and $u$ is finite and $1$-finely
	continuous at every point in $V:=U\setminus N$.
	By Theorem \ref{thm:quasiopen and finely open}, we can assume that $V$ is $1$-finely open.
	Let $\{(a_j,b_j)\}_{j=1}^{\infty}$ be an enumeration of all intervals in $\R$
	with rational endpoints
	and let
	\[
	V_j:=\{x\in V:\,a_j<u(x)<b_j\}.
	\]
	By the $1$-fine continuity of $u$, the sets $V_j$ are $1$-finely open.
	Hence by Theorem \ref{thm:quasiopen and finely open}, they are also $1$-quasiopen.
	Fix $\eps>0$. There are open sets $G_j\subset X$ such that
	$\capa_1(G_j)<2^{-j-1}\eps$ and each $V_j\cup G_j$ is open.
	Since $\capa_1$ is an outer capacity, there is also an open set
	$G_N\supset N$ such that $\capa_1(G_N)<\eps/2$. Now
	\[
	G:=G_N\cup \bigcup_{j=1}^{\infty}G_j
	\]
	is an open set such that $\capa_1(G)<\eps$, and $u|_{U\setminus G}$ is continuous
	since $V_j\cup G$ are open sets.
	
	To prove the converse direction, by Theorem
	\ref{thm:quasiopen and finely open} we know that $U=V\cup N$, where $V$ is
	$1$-finely open and $\mathcal H(N)=0$, and then also
	$\capa_1(N)=0$ by \eqref{eq:null sets of Hausdorff measure and capacity}.
	By the quasicontinuity of $u$, for each
	$j\in\N$ we find an open set $G_j\subset X$ such that $\capa_1(G_j)<1/j$ and
	$u|_{V\setminus G_j}$ is continuous. By \eqref{eq:capacity of fine closure},
	we have $\capa_1(\overline{G_j}^1)=\capa_1(G_j)$ for each $j\in\N$,
	and so the set
	\[
	A:=N\cup \bigcap_{j=1}^{\infty}\overline{G_j}^1
	\]
	satisfies $\capa_1(A)=0$. If $x\in U\setminus A$, then
	$x\in V\setminus \overline{G_j}^1$ for some $j\in\N$.
	Since $V\setminus \overline{G_j}^1$ is a $1$-finely open set and
	$u|_{V\setminus \overline{G_j}^1}$ is continuous, it follows that
	$u$ is finite and $1$-finely continuous at $x$.
\end{proof}

We will need the following quasi-Lindel\"of principle from \cite{L-CK}.
\begin{theorem}[{\cite[Theorem 5.2]{L-CK}}]\label{thm:quasi-Lindelof}
	For every family $\mathcal V$ of $1$-finely open sets there is a countable subfamily
	$\mathcal V'$ such that
	\[
	\capa_1\left(\bigcup_{V\in\mathcal V} V\setminus \bigcup_{V'\in\mathcal V'}V'\right)=0.
	\]
\end{theorem}

From now on, $U\subset X$ is always a $1$-quasiopen set.
Note that $1$-quasiopen sets are $\mu$-measurable by \cite[Lemma 9.3]{BB-OD}.

\begin{definition}
A family $\mathcal B$ of $1$-quasiopen sets is a \emph{1-quasicovering} of $U$ if it is countable,
$\bigcup_{V\in \mathcal B}V\subset U$, and
$\capa_1\left(U\setminus \bigcup_{V\in \mathcal B}V\right)=0$.
If every $V\in\mathcal B$ is a $1$-finely open $1$-strict subset of $U$ with
$V\Subset U$, then $\mathcal B$ is a \emph{1-strict quasicovering} of $U$.
Moreover, we say that
\begin{enumerate}
\item $u\in N^{1,1}_{\mathrm{fine-loc}}(U)$ if $u\in N^{1,1}(V)$ for every $1$-finely open
$1$-strict subset $V\Subset U$;
\item $u\in N^{1,1}_{\mathrm{quasi-loc}}(U)$ (respectively $L^1_{\mathrm{quasi-loc}}(U)$)
if there is a $1$-quasicovering $\mathcal B$ of $U$ such that $u\in N^{1,1}(V)$
(respectively $L^1(V)$)
for every $V\in\mathcal B$.
\end{enumerate}
\end{definition}

\begin{proposition}\label{prop:existence of strict quasicovering}
There exists a $1$-strict quasicovering $\mathcal B$ of $U$.
Moreover, the associated Newton-Sobolev functions can be chosen compactly supported in $U$.
\end{proposition}

\begin{proof}
By Theorem \ref{thm:quasiopen and finely open}, we have $U=V\cup N$, where $V$ is $1$-finely open and
$\mathcal H(N)=0$, and then also $\capa_1(N)=0$
by \eqref{eq:null sets of Hausdorff measure and capacity}.
For every $x\in V$, by Theorem \ref{thm:strict subsets with Newtonian functions}
we find a $1$-finely open set
$V_x\ni x$ such that $V_x\Subset V$ and an associated function
$v_x\in N_0^{1,1}(V)$ such that $0\le v_x\le 1$ on $X$,
$v_x=1$ on $V_x$, and $\supp v_x\Subset V$.
The collection $\mathcal B':=\{V_x\}_{x\in V}$ covers $V$,
and by the quasi-Lindel\"of principle (Theorem \ref{thm:quasi-Lindelof})
and the fact that $\capa_1(U\setminus V)=0$, there exists
a countable subcollection $\mathcal B\subset \mathcal B'$ such that
$\capa_1\left(U\setminus \bigcup_{V_x\in \mathcal B}V_x\right)=0$.
\end{proof}

It follows that $N^{1,1}_{\mathrm{fine-loc}}(U)\subset N^{1,1}_{\mathrm{quasi-loc}}(U)$.
From now on, since the proofs given in \cite{BBL-SS} in the case $1<p<\infty$
apply almost verbatim also in our setting, we will only
point out the differences with \cite{BBL-SS}.

\begin{theorem}\label{thm:quasiloc is finely and quasicontinuous}
Let $u\in N^{1,1}_{\mathrm{quasi-loc}}(U)$. Then $u$ if finite $1$-q.e. and $1$-finely continuous
$1$-q.e. on $U$. Thus $u$ is also $1$-quasicontinuous on $U$.
\end{theorem}

\begin{proof}
Follow verbatim the proof of \cite[Theorem 4.4]{BBL-SS}, except that replace the reference to
\cite[Proposition 4.2]{BBL-SS} by Proposition \ref{prop:existence of strict quasicovering}, and the references to \cite[Theorem 4.9(b)]{BBL-WC} and
\cite[Theorem 1.4(b)]{BBL-CCK}  by
Theorem \ref{thm:fine continuity and quasicontinuity equivalence}.
\end{proof}

\begin{definition}
A nonnegative function $\widetilde{g}_u$ on $U$ is a
\emph{1-fine upper gradient} of
$u\in N^{1,1}_{\mathrm{quasi-loc}}(U)$ if there is a quasicovering $\mathcal B$ of $U$ such that for every $V\in\mathcal B$,
$u\in N^{1,1}(V)$  and $\widetilde{g}_u=g_{u,V}$ a.e. on $V$,
where $g_{u,V}$ is the minimal $1$-weak upper gradient of $u$ in $V$.
\end{definition}

\begin{lemma}
If $u\in N^{1,1}_{\mathrm{quasi-loc}}(U)$, then it has a unique
(in the a.e. sense) $1$-fine upper gradient
$\widetilde{g}_u$.
\end{lemma}

\begin{proof}
Follow verbatim the proof of \cite[Lemma 5.2]{BBL-SS}.
\end{proof}

\begin{theorem}\label{thm:fine ug is weak ug}
If $u\in N^{1,1}_{\mathrm{quasi-loc}}(U)$ and $\widetilde{g}_u$ is a $1$-fine upper gradient
of $u$, then $\widetilde{g}_u$ is a $1$-weak upper gradient of $u$ in $U$.
\end{theorem}

\begin{proof}
Follow verbatim the proof of \cite[Theorem 5.3]{BBL-SS}.
\end{proof}

\begin{proposition}\label{prop:existence of extensions}
If $u\in N^{1,1}_{\mathrm{quasi-loc}}(U)$, then there is a $1$-strict quasicovering $\mathcal B$
of $U$ such that for every $V\in \mathcal B$, there exists $u_V\in N^{1,1}(X)$ with
$u=u_V$ on $V$.
\end{proposition}
\begin{proof}
Follow verbatim the proof of \cite[Proposition 5.5]{BBL-SS}, except that replace
the reference to
\cite[Theorem 4.4]{BBL-SS} by
Theorem \ref{thm:quasiloc is finely and quasicontinuous}, and
\cite[Proposition 4.2]{BBL-SS} by Proposition \ref{prop:existence of strict quasicovering}.
\end{proof}

The following definition is originally from Kilpel\"ainen--Mal\'y
\cite{KiMa}.

\begin{definition}
Let $U\subset \R^n$. A function $u\in L^1(U)$ is in $W^{1,1}(U)$ if
\begin{enumerate}
\item there is a quasicovering $\mathcal B$ of $U$ such that for every $V\in\mathcal B$
there is an open set $G_V\supset V$ and $u_V\in W^{1,1}(G_V)$ such that
$u=u_V$ on $V$, and
\item the \emph{fine gradient} $\nabla u$, defined by $\nabla u=\nabla u_V$ a.e. on each
$V\in \mathcal B$, also belongs to $L^1(U)$.
\end{enumerate}
Moreover, let
\[
\Vert u\Vert_{W^{1,1}(U)}:=\int_U (|u|+|\nabla u|)\,dx.
\]
\end{definition}

Recall that we constantly assume $U$ to be a $1$-quasiopen set.

\begin{theorem}\label{thm:N is same as W}
Let $U\subset \R^n$. Then $u\in W^{1,1}(U)$ if and only if there exists
$v\in N^{1,1}(U)$ such that $v=u$ a.e. on $U$. Moreover, $g_{v}=|\nabla u|$ a.e. on $U$
and $\Vert v\Vert_{N^{1,1}(U)}=\Vert u\Vert_{W^{1,1}(U)}$.
\end{theorem}

Here $g_v$ is the minimal $1$-weak upper gradient of $v$ in $U$.

\begin{proof}
Follow verbatim the proof of \cite[Theorem 5.7]{BBL-SS}, except that replace the reference to
\cite[Proposition 5.5]{BBL-SS} by Proposition \ref{prop:existence of extensions},
\cite[Proposition A.12]{BB} by \cite[Corollary A.4]{BB},
and \cite[Theorem 5.4]{BBL-SS} by Theorem \ref{thm:fine ug is weak ug}.
\end{proof}

Returning momentarily to the metric space setting, define the space
\[
\widehat{N}^{1,1}(U):=\{u:\,u=v \textrm{ a.e. for some }v\in N^{1,1}(U)\}.
\]

\begin{theorem}\label{thm:quasicontinuous representatives}
Let $u\in \widehat{N}^{1,1}(U)$. Then $u\in N^{1,1}(U)$ if and only
if $u$ is $1$-quasicontinuous on $U$.
\end{theorem}
\begin{proof}
Assume that $u$ is $1$-quasicontinuous on $U$. There is $v\in N^{1,1}(U)$ such that $u=v$ a.e. on $U$.
By Theorem \ref{thm:quasiloc is finely and quasicontinuous}, $v$ is $1$-quasicontinuous on $U$.
By \cite[Proposition 5.23]{BB} and \cite[Proposition 4.2]{BBM},
$u=v$ $1$-q.e. on $U$, and thus $u\in N^{1,1}(U)$
by \eqref{eq:quasieverywhere equivalence classes}.

The converse follows from Theorem \ref{thm:quasiloc is finely and quasicontinuous}.
\end{proof}

\begin{theorem}
Let $U\subset \R^n$, and let $u$ be an everywhere defined function
on $U$. Then $u\in N^{1,1}(U)$ if and only if $u\in W^{1,1}(U)$ and $u$
is $1$-quasicontinuous. Moreover, then
$\Vert u\Vert_{N^{1,1}(U)}=\Vert u\Vert_{W^{1,1}(U)}$.
\end{theorem}
\begin{proof}
This follows from Theorems \ref{thm:N is same as W} and
\ref{thm:quasicontinuous representatives}.
\end{proof}

\noindent Address:\\

\noindent University of Jyvaskyla\\
Department of Mathematics and Statistics\\
P.O. Box 35, FI-40014 University of Jyvaskyla, Finland\\
E-mail: {\tt panu.k.lahti@jyu.fi}

\end{document}